\documentclass{article}
\usepackage{subfig}
\usepackage[section] {placeins}
\usepackage{amsthm}
\usepackage{amssymb,amsmath}
\usepackage{anysize}
\usepackage{psfrag}
\newtheorem{thm}{Theorem}[section]
\newtheorem{def.}{Definition}[section]

\newtheorem{prop}{Proposition}[section]
\newtheorem{cor}{Corollary}[section]

\newtheorem{lem}{Lemma}[section]

\numberwithin{table}{section}
\begin{document}

\title{The Minimization of the Number of Colors is Different at p=11}
\author{Pedro Lopes\\
        Center for Mathematical Analysis, Geometry, and Dynamical Systems\\
        Department of Mathematics\\
        Instituto Superior T\'ecnico, University  of Lisbon\\
        Av. Rovisco Pais\\
        1049-001 Lisbon\\
        Portugal\\
        \texttt{pelopes@math.tecnico.ulisboa.pt}\\
}
\date{April 08, 2015}
\maketitle

\begin{abstract}
In this article we present the following new fact for prime $p=11$. For knots $6_2$ and $7_2$, $mincol_{11}\, 6_2 = 5 = mincol_{11}\, 7_2$, along with the following feature. There is a pair of diagrams, one for $6_2$ and the other one for $7_2$, each of them admitting only non-trivial $11$-colorings using $5$ colors, but neither of them admitting being colored with the sets of $5$ colors that color the other one.  This is in full contrast with the behavior exhibited by links admitting non-trivial $p$-colorings over the smaller primes, $p=2, 3, 5$ or $7$.

We also prove results concerning obstructions to the minimization of colors over generic odd moduli. We apply these to find the right colors to eliminate from non-trivial colorings. We thus prove that $5$ is the minimum number of colors for each knot of prime determinant $11$ or $13$ from Rolfsen's table.
\end{abstract}

\bigbreak


\bigbreak


\bigbreak

\section{Introduction}

\noindent

In this article we present the following new fact for prime $p=11$. There are two knots, $6_2$ and $7_2$, with $mincol_{11}\, 6_2 = 5 = mincol_{11}\, 7_2$, along with the following feature. There is a pair of diagrams, one for $6_2$ and the other one for $7_2$, each of them admitting only non-trivial $11$-colorings using $5$ colors, but neither of them admitting being colored with the sets of $5$ colors that color the other one (see Figures \ref{fig:6-2} and \ref{fig:7-2}). This is in full contrast with the behavior exhibited by links admitting non-trivial $p$-colorings over the smaller primes, $p=2, 3, 5$ or $7$ and leads to the main result of this article, Theorem \ref{thm:nounivered}.

This article is organized as follows. In Section \ref{sect:back} the background material is introduced. In Section \ref{sect:remarkable} we prove the fact referred to above. This suggests new objects and terminology which we develop in Subsection \ref{subsect:new}. In Section \ref{sect:obst} we prove facts pertaining to obstructions to non-trivial colorings. In Section \ref{sect:11-13} we apply these facts to minimize the number of colors of knots of determinant $11$ and $13$ found in Rolfsen's tables (\cite{Rolfsen}) along with $T(2, 11)$ and $T(2, 13)$. We thus prove that for all these knots but $T(2, 13)$, $5$ is the minimum number of colors modulo the determinant of the knot at issue. These results are collected in Tables \ref{Ta:mincolsub11} and \ref{Ta:mincolsub13} (the new terminology used in these Tables is defined in Subsection \ref{subsect:new}). We conclude the article with a discussion of future work in Section \ref{sect:future}.
\begin{table}[h!]
\begin{center}
\begin{tabular}{| c ||   c |  c |}\hline
$L$                      & $mincol_{11} L$ & $11$-minimal sufficient set of colors for $L$ \\ \hline \hline
$6_2$ $\ast$             & $5$            &  $\{ 0, 2, 3, 4, 8 \}$ \quad (Fig. \ref{fig:6-2})     \\ \hline
$7_2$        & $5$            & $\{ 0, 3, 4, 5, 6 \}$ \quad (Fig. \ref{fig:7-2})  \\ \hline
$10_{125}$ $\ast$        & $5$           & $\{ 0, 1, 5, 8, 10 \}$ \quad (Figs. \ref{fig:10-125}, \ref{fig:10-125cont}, and \ref{fig:10-125contcont})\\ \hline
$10_{128}$ $\ast$        & $5$           & $\{ 2, 3, 4, 6, 9 \}$ \quad (Figs. \ref{fig:10-128}, \ref{fig:10-128concl})\\ \hline
$10_{152}$ $\ast$        & $5$           & $\{ 0, 2, 3, 4, 8 \}$  \quad (Fig. \ref{fig:10-152})\\ \hline
$T(2, 11)$   & $5$             & $\{ 0, 1, 2, 3, 6 \}$ \quad (\cite{klgame}, Fig. $3$)\\ \hline
\end{tabular}
\caption{$mincol_{11}L$ for the  knots of prime determinant $11$  from Rolfsen's tables (\cite{Rolfsen}) and for $T(2, 11)$ (\cite{klgame}).  Each pair (knot with a $\ast$, a knot without a $\ast$) presents the following feature. There is a pair of diagrams, one for each knot in the pair of knots, such that  each diagram admits non-trivial $11$-colorings with $5$ colors, but none of these diagrams is colored with the $5$ colors that color the other diagram. Below we will subsume this by saying they do not admit the same $11$-minimal sufficient set of colors (see Definition \ref{def:minsuf}).}
\label{Ta:mincolsub11}
\end{center}
\begin{center}
\begin{tabular}{| c ||   c | c | }\hline
$L$         &  $mincol_{13} L$ &  $13$-minimal sufficient set of colors for $L$ \\ \hline \hline
$6_3$       &   $5$            &  $\{ 0, 1, 2, 6, 11 \}$\quad (Fig. \ref{fig:6-3}) \\ \hline
$7_3$       &   $5$            &  $\{ 0, 2, 3, 4, 9 \}$\quad (Figs. \ref{fig:7-3} and \ref{fig:7-3concl})   \\ \hline
$8_1$       &   $5$       &  $\{ 0, 1, 3, 8, 12 \}$ \quad (Fig. \ref{fig:8-1new})\\ \hline
$9_{43}$    &   $5$       &  $\{ 0, 1, 3, 8, 12 \}$  \quad (Figs. \ref{fig:9-43} through \ref{fig:9-43IIIv2})\\ \hline
$10_{154}$  &   $5$            &  $\{ 0, 3, 4, 5, 10 \}$\quad (Figs. \ref{fig:10-154} and \ref{fig:10-154cont}) \\ \hline
$T(2, 13)$  &   $\leq 6$       &  $\subseteq \{ 0, 1, 2, 3, 4, 7 \}$ {\bf?} \quad (\cite{klgame}, Figs. $4$, $5$, and $6$)  \\ \hline
\end{tabular}
\caption{$mincol_{13}L$ for the  knots of prime determinant $13$  from Rolfsen's tables (\cite{Rolfsen}) and for $T(2, 13)$ (\cite{klgame}). At the time of writing it was not possible to ascertain if  $mincol_{13}\, T(2, 13) = 5$. Modulo $13$, each set of $5$ colors which colors non-trivially a knot diagram, will also color non-trivially any other diagram admitting a non-trivial $13$-coloring using $5$ colors, see \cite{Ge5} and \cite{nakamuranakanishisatoh}.}
\label{Ta:mincolsub13}
\end{center}
\end{table}

\bigbreak
{\bf Remark} A boxed integer at the top left of a Figure indicates the modulus with respect to which the colorings in the Figure are being considered.

\bigbreak

\subsection{Acknowledgements}\label{subsect:orgackn}

\noindent

The author thanks Louis H. Kauffman and Slavik Jablan for remarks and suggestions.

The author acknowledges support from FCT (Funda\c c\~ao para a Ci\^encia e a Tecnologia), Portugal, through project number PTDC/MAT/101503/2008, ``New Geometry and Topology''.

\section{Background Material}\label{sect:back}

\noindent

We begin by recalling the definition of a (Fox) coloring of a link along with related objects and results. A finer analysis is found in \cite{GJKL}. Consider a link along with one of its diagrams. Regard the arcs of the diagram as algebraic variables and at each crossing read off the {\bf coloring condition}: twice the ``over-arc'' minus the ``under-arcs'' equals zero, see Figure \ref{fig:xtop}.


\begin{figure}[!ht]
	\psfrag{ai}{\Huge$a$}
	\psfrag{ai+1}{\Huge$c$}
	\psfrag{aji}{\Huge$b$}
	\psfrag{eq}{\Huge$c=2b-a$}
	\centerline{\scalebox{.33}{\includegraphics{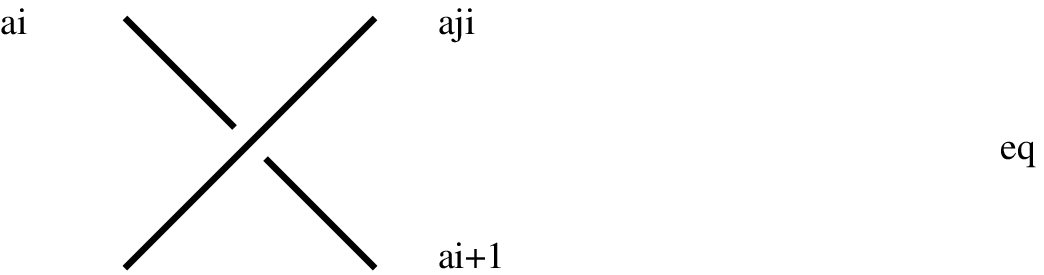}}}
	\caption{The coloring condition (in an equivalent form), $c=2b-a$, relating the arcs that meet at a crossing.}\label{fig:xtop}
\end{figure}
Consider the system of linear homogeneous equations that we so associate to the link diagram. The matrix of the coefficients of this system of equations is called the \emph{coloring matrix of the diagram}. The determinant of the coloring matrix is zero since along each row the non-null entries are one $2$ (the ``over-arc contribution'') and two $-1$'s (the ``under-arc contributions''). Thus there are always solutions to this system of equations for which any two variables take on the same value; these are called the trivial colorings. The other solutions are called non-trivial colorings.

In \cite{GJKL}, Ge, Jablan, Kauffman, and the author proved that, for each odd prime modulus $p$, there is an equivalence relation among the set of non-trivial $p$-colorings of a diagram, and that the number of these equivalence classes constitute a topological invariant for the link at stake. This equivalence relation is brought about by the so-called $p$-coloring automorphisms. Here are the relevant definitions and results.

\begin{def.}\label{def:colauto}$(\cite{elhamdadi})$
Let $p$ be an odd prime. A permutation, $f$, of the integers mod $p$ i.e., of the set $\mathbf{Z}/p\mathbf{Z}$,  is called a $(p$-$)${\it coloring automorphism} if for any two integers $a$ and $b$, the coloring condition is preserved i.e.,
\[
f(a\ast b) = f(a)\ast f(b) \qquad \text{ mod } p
\]
where $a\ast b:= 2b-a$.
\end{def.}

\begin{prop}\label{prop:colautopreservecol}
Coloring automorphisms preserve colorings.

Specifically, let $D$ be a link diagram with arcs $a_1, \dots, a_n$ as we travel along the diagram. Let $f$ be a $p$-coloring automorphism. If $D$ admits a $p$-coloring such that for each $i$, arc $a_i$ receives color $c_i$, then the assignment ``arc $a_i$ receives color $f(c_i)$'', for each $i$, is also a $p$-coloring of $D$.
\end{prop}
\begin{proof}
We keep the notation and terminology of the statement with the following addition. We let $a_{j_i}\in \{ a_1, \dots, a_n \}$ denote the over-arc at the crossing where arcs $a_i$ and $a_{i+1}$ meet.

Since the assignment ``arc $a_i$ receives color $c_i$'', for each $i$, is a $p$-coloring, then $c_{i+1}=2c_{j_i}-c_i$ (mod $p$) for each $i$. Thus
\[
f(c_{i+1})=f(2c_{j_i}-c_{i})=2f(c_{j_i})-f(c_{i}) \qquad (\text{mod } p)
\]
for each $i$ i.e., the assignment arc ``$a_i$ receives color $f(c_i)$'' for each $i$ is also a $p$-coloring of $D$. This completes the proof.
\end{proof}

\begin{prop}\label{prop:colautoform} $(\cite{elhamdadi})$
Given an odd prime $p$, a $p$-coloring automorphism is of the form:
\[
f_{\lambda, \mu}(x)=\lambda x +\mu
\]
where $\lambda $ is any  element from the set of units of the integers mod $p$ and $\mu$ is any integer mod $p$. Furthermore, the set of $p$-coloring automorphisms equipped with composition of functions for binary operation constitutes a group.
\end{prop}
\begin{proof}The proof is found in \cite{elhamdadi}.
\end{proof}

\begin{def.}\label{def:equivclasscol}$(\cite{GJKL})$ Let $p$ be a prime. Let $L$ be a link admitting non-trivial $p$-colorings and $D$ be one of its diagrams. Let $\cal C$ and $\cal C'$ be non-trivial $p$-colorings of $D$ i.e., $\cal C$ and $\cal C'$ are mappings from the arcs of $D$ to $\mathbf{Z}/p\mathbf{Z}$ such that if $a$, $b$, and $c$ are the arcs as in Figure \ref{fig:xtop} then $2{\cal C} (b) -{\cal C}(a)={\cal C}(c)$ and $2{\cal C'} (b) -{\cal C'}(a)={\cal C'}(c)$ $($note that here we let $a, b,$ and $c$ stand for the arcs$)$.
Then $\cal C$ is related to $\cal C'$ if, by definition, there is a $p$-coloring automorphism, $f$, such that,
\[
{\cal C'} (a) = f({\cal C}(a))
\]
for any arc $a$ of the diagram $D$.
\end{def.}

\begin{thm}\label{thm:preliminaries} Let $L$ be a link and $D$ one of its diagrams.
\begin{enumerate}
\item The equivalence class of the coloring matrix of $D$ under elementary transformations of matrices is a topological invariant of $L$.
\item The Smith Normal Form from this equivalence class, call it $SNF_L$, possesses at least one $0$ along its diagonal $($stemming from the determinant of this matrix being $0 )$; the product of the diagonal entries of $SNF_L$ except the $0$ referred to above, is a topological invariant of $L$ called the \emph{determinant of $L$}, notation $\det L$.
\bigbreak
Assume further that $p$ is prime.
\bigbreak
\item Let $n_p(L)$ be the number of $0$'s, mod $p$, along the diagonal of $SNF_L$. If $n_p(L) \geq 2$ $($or, equivalently, if $p$ divides $\det L )$ then $L$ admits non-trivial colorings modulo $p$ $($also known as non-trivial $p$-colorings$)$. $n_p(L)$ is a topological invariant for $L$.
\item The relation set forth in Definition \ref{def:equivclasscol} is an equivalence relation among the set of non-trivial $p$-colorings on $D$. Assuming $n_p(L) \geq 2$, there are $\frac{p^{n_p(L)-1}-1}{p-1}$ equivalence classes of non-trivial $p$-colorings on $D$. This number is a topological invariant for $L$ i.e., no matter which diagram of $L$ is used, it has $\frac{p^{n_p(L)-1}-1}{p-1}$ equivalence classes of non-trivial $p$-colorings associated to it.
\end{enumerate}
\end{thm}
\begin{proof} See \cite{GJKL}.
\end{proof}

\begin{def.}\label{def:pnullitylink}$[ p$-nullity of a link $]$
Keeping the notation from Theorem \ref{thm:preliminaries}, $n_p(L)$ is called the $p$-nullity of $L$.

We remark that $n_p(L) \geq 1$.

If a link $L$ has prime determinant $p$, then $n_p(L) = 2$ which further implies that $\frac{p^{n_p(L)-1}-1}{p-1} = 1$ i.e., such a link only has $1$ equivalence class of non-trivial $p$-colorings.
\end{def.}

Statement $4$. in Theorem \ref{thm:preliminaries} is the reason why we do not stop extracting information from the coloring matrix at the determinant of the link under consideration (this determinant can be obtained without resorting to the SNF).

For instance, $\det\,  \big( 9_{49}\big) = 25 = \det\,  \big( 10_3 \big) $, but $n_5\,  \big( 9_{49} \big)  = 3 \neq 2 = n_5\,  \big( 10_3 \big) $.

In fact the determinant of the link tells us if a modulus will bring about non-trivial colorings for this link. But if prime $p$ divides the determinant of the link, the $p$-nullity of the link tells us further into how many equivalence classes the non-trivial $p$-colorings of the link split into. And this will be crucial for the argument of our main result here. The pair of knots on which we build our argument, $6_2$ and $7_2$, both have prime determinant, $11$, which means there is only one equivalence class of non-trivial $11$-colorings for each of these knots.

\begin{cor}\label{cor:preliminaries}
We keep the notation of Theorem \ref{thm:preliminaries} along with the terminology from Definition \ref{def:pnullitylink}. If $n_p(L) \geq 2$ then $L$  has $\frac{p^{n_p(L)-1}-1}{p-1}$ equivalence classes of $p$-colorings and the group of $p$-coloring automorphisms acts transitively on each of these equivalence classes. If the same diagram is supporting two $p$-colorings from a given equivalence class, then they use the same number of distinct colors, and one set of these colors is the image of the other one under a $p$-coloring automorphism.
\end{cor}
\begin{proof} See \cite{GJKL}.
\end{proof}

{\bf Remark} Split links (which have null determinant) are  non-trivially colored with two colors in any modulus. In general links of null determinant admit colorings in any modulus and we believe they deserve an independent study which we plan on doing in a separate article. For this reason all links referred to in this article are links of non-null determinant.

{\bf Remark} We are aware of the connection between the coloring matrix and the first homology group of the $2$-fold branched cover of $S^3$ along the link. We expect to explore this connection in future work.
\bigbreak

If there is a prime $p$ such that  the link $L$ under study has $p$-nullity $n_p(L) \geq 2$ then $L$ admits non-trivial $p$-colorings. In such circumstance we can look into the minimum number of colors it takes to assemble a non-trivial $p$-coloring of $L$.
\begin{def.}$($minimum number of colors; minimal coloring$)$\label{def:mincol}
Let $p$ be a prime, $L$ a link admitting non-trivial $p$-colorings, and $D$ a diagram of $L$. Let $c_{p, D}$ stand for the minimum number of colors it takes to assemble a non-trivial $p$-coloring on $D$. We denote
\[
mincol_p\, L :=min \{ \,  c_{p, D} \, \,  | \,\,  D \text{ is a diagram of } L \, \}
\]
and call it {\bf the minimum number of colors for $L$, modulo $p$}.
\bigbreak
We call {\bf minimal  (p$-$)coloring of $L$} any non-trivial $p$-coloring for a diagram of $L$ using $mincol_p\, L$ colors.
\end{def.}

This notion was first introduced in \cite{Frank}. We remark that the minimum number of colors is a link invariant. For each of the first primes ($p=2, 3, 5$, or $7$) the pattern of the minimum number of colors mod $p$ is very rigid. The precise mathematical content of this statement is Theorem \ref{thm:p<11} below. We let $n\mid m$ stand for ``integer $n$ divides integer $m$''.

\begin{thm}\label{thm:p<11}$\cite{kl, lm, Oshiro, Saito, satoh}$
Let $L$ be a link with $\det L \neq 0$.
\begin{enumerate}
\item If $2\mid \det L$ then $mincol_2\, L = 2$.

Any $2$ colors from $\{ 0, 1 \}$ $($mod $2)$ can be used to assemble a minimal $2$-coloring.
\item If $3\mid \det L$ then $mincol_3\, L = 3$.

Any $3$ colors from $\{ 0, 1, 2 \}$ $($mod $3)$ can be used to assemble a minimal $3$-coloring.
\item If $5\mid \det L$ then $mincol_5\, L = 4$.

Any $4$ colors from $\{ 0, 1, 2, 3, 4 \}$ $($mod $5)$ can be used to assemble a minimal $5$-coloring of $D$.
\item If $7\mid \det L$ then $mincol_7\, L = 4$.

Any $4$ colors from $\{ 0, 1, 2, 3, 4, 5, 6 \}$ $($mod $7)$, can be used to assemble a minimal $7$-coloring but the sets obtained from $\{ 0, 1, 2, 3  \}$ via a $7$-coloring automorphism.
\item Let $p$ be a prime greater than $7$. If $p\mid \det L$ then $mincol_p\, L \geq 5$.
\end{enumerate}
\end{thm}
\begin{proof}
The minimum number of colors in the first four instances and the estimate in the fifth instance are now  standard facts (see \cite{kl, lm, satoh, Oshiro, Saito}).

Now for the choice of colors. In each of the first two instances, there is no other possibility since the minimum number of colors coincides with the maximum number of colors.

In \cite{satoh} it is shown how any minimal $5$-coloring can be realized with colors $1, 2, 3, 4$.

In \cite{Oshiro} it is shown how any minimal $7$-coloring can be realized with colors $0, 1, 2, 4$.

Lemma \ref{lem:012} and Corollary \ref{cor:lem012} below conclude the proof, and this is the original part of this result.
\end{proof}

\begin{lem}[\cite{Saito}, \cite{Ge5}]\label{lem:012}
Let $p$ be an odd prime. If a non-trivial $p$-coloring of a diagram does not use colors $0, 1, 2$ mod $p$, then  there is an equivalent $p$-coloring which uses colors $0, 1, 2$. Therefore, without loss of generality, colors $0, 1, 2$ can be assumed to be part of the set of colors used by a non-trivial $p$-coloring.
\end{lem}
\begin{proof}
The proof of this result is found in \cite{Ge5}. We include it here for the convenience of the reader. Any diagram admitting a non-trivial coloring over an odd modulus has a crossing whose arcs bear distinct colors, say $a, b, c$ with $c=2b-a$ over the odd modulus (because if we try to use only two colors we will end up with either $a=b=c$ or an even modulus, \cite{kl}). Then $f(x)=(b-a)^{-1}(x-a)$ (mod $p$) is a $p$-coloring automorphism (see Proposition \ref{prop:colautopreservecol}) which transforms the set of colors used by the current non-trivial $p$-coloring into a new set of colors for a non-trivial $p$-coloring using colors $0, 1, 2$. This completes the proof.
\end{proof}

\begin{cor}\label{cor:lem012}
Any set of four colors mod $5$ can be used  on a minimal non-trivial $5$-coloring.

Any image of the set $\{ 0, 1, 2, 4 \}\,  ($mod $7 )$ via a $7$-coloring automorphism, can be used to color a minimal non-trivial $7$-coloring.
\end{cor}
\begin{proof}
Lemma \ref{lem:012} allows us to use colors $0, 1, 2$ without loss of generality. On the other hand we already know that the minimum number of colors mod $5$ or mod $7$ is four. In each of these cases we just have to specify the fourth color.

Let us look at the mod $5$ case. The fourth color is either $3$ or $4$. But the $5$-coloring automorphism $f(x)=2x$ (mod $5$) maps $\{ 0, 1, 2, 3  \}$ into $\{  0, 1, 2, 4 \}$. Hence any set of four colors mod $5$ is the image of the set $\{ 0, 1, 2, 3 \}$ and can be used to color any diagram admitting minimal non-trivial $5$-colorings.

Now for mod $7$. The sets $\{ 0, 1, 2, 3 \}$ and $\{ 0, 1, 2, 6 \}$ ($f(x)=x+6$ mod $7$ transforms one into the other) are not admissible as we will see below in Theorem \ref{thm:lessk} and Corollary \ref{cor:s-consec}. The fourth color can then be either $4$, or $5$ and $f(x)=3x+2$ mod $7$ sends $\{  0, 1, 2, 4 \}$ into $\{ 0, 1, 2, 5\}$ mod $7$.

This concludes the proof.
\end{proof}

We extract the following Corollary from Theorem \ref{thm:p<11}.

\begin{cor}\label{cor:rigiditysmallerprimes}
Let $p$ be a prime from $\{ 2, 3, 5, 7 \}$. There exists a set, call it $S_p$, of colors mod $p$ whose cardinality is $2$ if $p=2$, $3$ if $p=3$, and $4$ if $p=5$ or $7$. Moreover, for any link $L$ admitting non-trivial $p$-colorings, a minimal $p$-coloring of $L$ can be set up using the colors from $S_p$.
\end{cor}
\begin{proof}
For $p=2$ or $p=3$ the statement is obvious since the minimum number of colors coincides with the maximum number of colors for these primes. For primes $p=5$ or $p=7$, Corollary \ref{cor:lem012} shows the statement is also true.
\end{proof}

\section{A New Fact For Prime $p=11$}\label{sect:remarkable}

\noindent

\begin{thm}\label{thm:6-2-7-2}
Knots $6_2$ and $7_2$ both have determinant $11$ and $mincol_{11}\, 6_2 = 5 = mincol_{11}\, 7_2$. Furthermore, there is a diagram for $6_2$, call it $D_{6_2}$, which admits a non-trivial $11$-coloring using $5$ colors and there is a diagram for $7_2$, call it $D_{7_2}$, which admits a non-trivial $11$-coloring using $5$ colors. But any set of $5$ colors used on $D_{6_2}$ to assemble a non-trivial $11$-coloring will not give rise to a non-trivial $11$-coloring on $D_{7_2}$, and reciprocally.
\end{thm}
\begin{proof}
Both knots $6_2$ and $7_2$ have determinant $11$, see \cite{Rolfsen}. In Figures \ref{fig:6-2} and \ref{fig:7-2} below we see that each of them can be non-trivially colored with as few as $5$ colors, mod $11$. Then $mincol_{11}\,  6_2 = 5 = mincol_{11}\,  7_2$, by Theorem \ref{thm:p<11}, statement $5$. The diagram on the right of Figure \ref{fig:6-2} is the $D_{6_2}$ referred to in the statement. Analogously, the diagram on the right of Figure \ref{fig:7-2} is the $D_{7_2}$. Furthermore,  we see that $D_{6_2}$ uses colors $\{ 0, 2, 3, 4, 8 \}$, mod $11$, whereas  $D_{7_2}$ uses colors $\{ 0, 3, 4, 5, 6 \}$, mod $11$.
\begin{figure}[!ht]
	\psfrag{0}{\huge$0$}
	\psfrag{1}{\huge$1$}
	\psfrag{2}{\huge$2$}
	\psfrag{3}{\huge$3$}
	\psfrag{4}{\huge$4$}
	\psfrag{8}{\huge$8$}
	\psfrag{11}{\huge$\mathbf{11}$}
	\centerline{\scalebox{.26}{\includegraphics{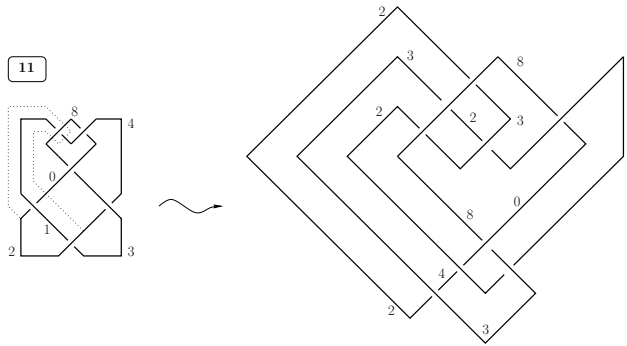}}}
	\caption{The knot $6_2$ whose determinant is $11$.   On the left-hand side, a diagram with minimum number of crossings equipped with a non-trivial $11$-coloring; the dotted lines indicate the move that will take it to the diagram on the right-hand side. On the right-hand side the diagram $D_{6_2}$ equipped with a non-trivial $11$-coloring using a minimum number of colors modulo $11$: $\{ 0, 2, 3, 4, 8 \}$.}\label{fig:6-2}
\end{figure}
\begin{figure}[!ht]
	\psfrag{0}{\huge$0$}
	\psfrag{1}{\huge$1$}
	\psfrag{2}{\huge$2$}
	\psfrag{3}{\huge$3$}
	\psfrag{4}{\huge$4$}
	\psfrag{5}{\huge$5$}
	\psfrag{6}{\huge$6$}
	\psfrag{11}{\huge$\mathbf{11}$}
	\centerline{\scalebox{.26}{\includegraphics{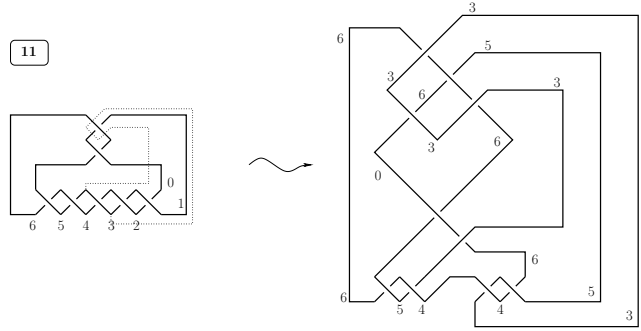}}}
	\caption{The knot $7_2$ whose determinant is $11$.   On the left-hand side, a diagram with minimum number of crossings equipped with a non-trivial $11$-coloring; the dotted lines indicate the move that will take it to the diagram on the right-hand side. On the right-hand side the diagram $D_{7_2}$ equipped with a non-trivial $11$-coloring using a minimum number of colors modulo $11$: $\{ 0, 3, 4, 5, 6 \}$.}\label{fig:7-2}
\end{figure}
We now prove that there is no $11$-coloring  automorphism that maps one of these sets into the other. We reason as follows. If there was such an automorphism, call it $f$, then $f(0)$ and $f(2)$ would be distinct elements of $\{  0, 3, 4, 5, 6 \}$. We then explore the different possibilities for the values $f(0)$ and $f(2)$ can take on and check that none of them gives rise to an automorphism that maps $\{ 0, 2, 3, 4, 8 \}$ onto $\{  0, 3, 4, 5, 6 \}$. For instance, in the first row of Table \ref{Ta:lambdax+mu} we assume $f(0)=0, f(2)=3$. Since $f(x)=\lambda x + \mu$ then $0=0\lambda  + \mu, 3=2\lambda +\mu $ so $\lambda = 7, \mu =0$. Then $f(3)=3\cdot 7 + 0=21=_{11}10$ (which does not belong to $\{  0, 3, 4, 5, 6 \}$ and this is indicated by the X), $f(4)=4\cdot 7 + 0=28=_{11}=6$, and $f(8)=8\cdot 7 +0=56=_{11}1$ (which does not belong to $\{  0, 3, 4, 5, 6 \}$ and this is indicated by the X). The calculations leading to the expressions in other rows of Table \ref{Ta:lambdax+mu} are performed analogously. In each row we see that there is always at least one element of $f(\{ 0, 2, 3, 4, 8 \})$ which does not belong to $\{  0, 3, 4, 5, 6 \}$. This completes the proof that there is no $11$-coloring automorphism which maps one of the indicated sets into the other.

Moreover, since both $6_2$ and $7_2$ have prime determinant $11$, then each of them admits only one equivalence class of $11$-colorings (see Theorem \ref{thm:preliminaries}, Statement $4$., and Definition \ref{def:pnullitylink}). This means that if a set of $5$ colors mod $11$, say $\{  a_1, \dots , a_5  \}$ could non-trivially color $D_{6_2}$ (respect., if a set of $5$ colors mod $11$, say $\{  b_1, \dots , b_5  \}$ could non-trivially color $D_{7_2}$), then there would be an $11$-coloring automorphism, say $f$, such that $f(\{  a_1, \dots , a_5  \}) = \{ 0, 2, 3, 4, 8 \}$ (respect., there would be an $11$-coloring automorphism, say $g$, such that $g(\{ 0, 3, 4, 5, 6 \}) = \{ b_1, \dots , b_5 \}$). If there was an $11$-coloring automorphism, call it $h$, such that $\{  a_1, \dots , a_5  \} = h(\{ b_1, \dots , b_5 \})$ then $f\circ h\circ g$ would be an $11$-coloring automorphism such that $\{ 0, 2, 3, 4, 8 \} = f(\{  a_1, \dots , a_5  \}) = f( h(\{ b_1, \dots , b_5 \})) = f(h(g(\{ 0, 3, 4, 5, 6 \})))$. But in the preceding paragraph  we proved that such an $11$-coloring automorphism does not exist. This concludes the proof.
\begin{table}[h!]
\begin{center}
\scalebox{.9}{\begin{tabular}{| c | c ||  c | c || c | c |  c | }\hline
$f(0)=0$ & $f(2)=3$  & $\lambda=7$ & $\mu=0$ & $f(3)=10$\, X & $f(4)=6$  & $f(8)=1$\, X  \\ \hline
$f(0)=0$ & $f(2)=4$  & $\lambda=2$ & $\mu=0$ & $f(3)=6$ & $f(4)=8$ \, X & $f(8)=5$  \\ \hline
$f(0)=0$ & $f(2)=5$  & $\lambda=8$ & $\mu=0$ & $f(3)=2$\, X & $f(4)=10$ \, X  & $f(8)=9$ \, X \\ \hline
$f(0)=0$ & $f(2)=6$  & $\lambda=3$ & $\mu=0$ & $f(3)=9$\, X & $f(4)=1$\, X  & $f(8)=2$\, X  \\ \hline \hline
$f(0)=3$ & $f(2)=0$  & $\lambda=4$ & $\mu=3$ & $f(3)=4$ & $f(4)=8$\, X  & $f(8)=2$\, X  \\ \hline
$f(0)=3$ & $f(2)=4$  & $\lambda=6$ & $\mu=3$ & $f(3)=10$ \, X & $f(4)=5$  & $f(8)=7$ \, X \\ \hline
$f(0)=3$ & $f(2)=5$  & $\lambda=1$ & $\mu=3$ & $f(3)=6$ & $f(4)=7$\, X   & $f(8)=0$  \\ \hline
$f(0)=3$ & $f(2)=6$  & $\lambda=7$ & $\mu=3$ & $f(3)=2$\, X & $f(4)=9$\, X  & $f(8)=4$  \\ \hline \hline
$f(0)=4$ & $f(2)=0$  & $\lambda=9$ & $\mu=4$ & $f(3)=9$\, X & $f(4)=7$\, X  & $f(8)=10$\, X  \\ \hline
$f(0)=4$ & $f(2)=3$  & $\lambda=5$ & $\mu=4$ & $f(3)=8$ \, X & $f(4)=2$\, X  & $f(8)=0$  \\ \hline
$f(0)=4$ & $f(2)=5$  & $\lambda=6$ & $\mu=4$ & $f(3)=0$ & $f(4)=6$   & $f(8)=8$\, X  \\ \hline
$f(0)=4$ & $f(2)=6$  & $\lambda=1$ & $\mu=4$ & $f(3)=7$\, X & $f(4)=8$\, X  & $f(8)=1$\, X  \\ \hline \hline
$f(0)=5$ & $f(2)=0$  & $\lambda=3$ & $\mu=5$ & $f(3)=3$ & $f(4)=6$  & $f(8)=7$\, X  \\ \hline
$f(0)=5$ & $f(2)=3$  & $\lambda=10$ & $\mu=5$ & $f(3)=2$ \, X & $f(4)=1$\, X  & $f(8)=8$\, X  \\ \hline
$f(0)=5$ & $f(2)=4$  & $\lambda=5$ & $\mu=5$ & $f(3)=9$\, X & $f(4)=3$   & $f(8)=1$\, X  \\ \hline
$f(0)=5$ & $f(2)=6$  & $\lambda=6$ & $\mu=5$ & $f(3)=1$\, X & $f(4)=7$\, X  & $f(8)=9$\, X  \\ \hline \hline
$f(0)=6$ & $f(2)=0$  & $\lambda=8$ & $\mu=6$ & $f(3)=8$\, X & $f(4)=5$  & $f(8)=4$  \\ \hline
$f(0)=6$ & $f(2)=3$  & $\lambda=4$ & $\mu=6$ & $f(3)=7$ \, X & $f(4)=0$  & $f(8)=5$  \\ \hline
$f(0)=6$ & $f(2)=4$  & $\lambda=10$ & $\mu=6$ & $f(3)=3$ & $f(4)=2$\, X   & $f(8)=9$\, X  \\ \hline
$f(0)=6$ & $f(2)=5$  & $\lambda=5$ & $\mu=6$ & $f(3)=10$\, X & $f(4)=4$  & $f(8)=2$\, X  \\ \hline
\end{tabular}}
\caption{There is no $f(x)=_{11}\lambda x + \mu$ such that  $f(\{ 0, 2, 3, 4, 8 \})=\{ 0, 3, 4, 5, 6 \}$. An $X$ indicates that the number to its left does not belong to $\{ 0, 3, 4, 5, 6 \}$.}
\label{Ta:lambdax+mu}
\end{center}
\end{table}
\end{proof}

\subsection{New Terminology}\label{subsect:new}
\noindent

Theorem \ref{thm:6-2-7-2} proves that the pattern described by Corollary \ref{cor:rigiditysmallerprimes} breaks down at $p=11$. We now introduce some new terminology. Among other things this new terminology allows us to state Theorem \ref{thm:nounivered} in a more concise form.

\begin{def.}$($sufficient set of colors; minimal sufficient set of colors$)$\label{def:minsuf}
Let $p$ be a prime and let $L$ be a link admitting non-trivial $p$-colorings.
\begin{itemize}
\item A {\bf $p$-sufficient set of colors} $($for $L$$)$ is a set of  integers mod $p$ such that a non-trivial $p$-coloring can be realized on some diagram of $L$ with colors from this set.
\item A {\bf $p$-minimal sufficient set of colors} $($for $L$$)$  is a $p$-sufficient set of colors $($for $L$$)$ whose cardinality is $mincol_p\, L$.
\end{itemize}
\end{def.}

This terminology refers to sets of colors capable of coloring particular diagrams of the link at stake. This terminology will be useful below when we try to minimize the colors of a coloring, see Section \ref{sect:11-13}. We now introduce the universal sort of object.

\begin{def.}$($universal minimal sufficient set of colors$)$\label{def:universalminsuf}
Let $p$ be a prime such that for any two links, $L$ and $L'$, admitting non-trivial  $p$-colorings, $mincol_p\, L = mincol_p\, L'$; in this case, let $mincol (p)$ denote this common minimum number of colors, $mincol (p) = mincol_p\, L = mincol_p\, L'$.

If there is a set, call it $S$, of $mincol(p)$ integers mod $p$ such that any diagram supporting minimal $p$-colorings can be colored with colors from this set, then we call $S$ a {\bf universal $p$-minimal sufficient set of colors}.

We say there is no universal $p$-minimal set of colors when either
\begin{itemize}
\item there are distinct links admitting non-trivial $p$-colorings  with different minimum numbers of colors; or
\item there is a unique minimum number of colors, $mincol(p)$, but there is no set of $mincol(p)$ distinct integers mod $p$ such that any diagram supporting a minimal $p$-coloring can be colored with the colors from this set.
\end{itemize}

\bigbreak

\end{def.}

We remark that Corollary \ref{cor:rigiditysmallerprimes} above proves that there is a universal $p$-minimal sufficient set of colors for each of the primes $p=2, 3, 5,$ and $7$. Furthermore, it proves that for any of the primes $2, 3, 5$, any set with $mincol (p)$ colors is a universal $p$-minimal sufficient set of colors, whereas for prime $7$, any set with $mincol (7)$ colors, but those obtained from $\{ 0, 1, 2, 3\}$ via a $7$-coloring automorphism,  is a universal $7$-minimal sufficient set of colors.
\bigbreak
We will drop the $p$- from these designations when it is clear from context which modulus we are working on.  With this terminology we can now state the following.

\begin{thm}\label{thm:nounivered}
There is no universal $11$-minimal sufficient set of colors.
\end{thm}
\begin{proof}
This is a restatement of Theorem \ref{thm:6-2-7-2} using the new terminology. Theorem \ref{thm:6-2-7-2} was already proved.
\end{proof}

\section{Obstructions to Minimizing the Set of Colors of a Coloring}\label{sect:obst}

\noindent

We prove that certain sets of colors cannot be sufficient sets of colors. This is helpful when choosing the next color to eliminate in the process of reducing the number of colors of a non-trivial coloring.

\begin{thm}\label{thm:lessk}
Let $k$ be a positive integer and $L$ a link with non-null determinant, admitting non-trivial $(2k+1)$-colorings.

If $S \subseteq \{ 0, 1, 2, \cdots , k  \}$ $($mod $2k+1$$)$, then $S$ is not a $(2k+1)$-Sufficient Set of Colors for $L$.
\end{thm}

\begin{proof}
Assume to the contrary and suppose there is a diagram of $L$ endowed with a non-trivial $(2k+1)$-coloring whose sequence of representatives of the distinct colors is $(c_i)\subseteq \{ 0, 1, 2, \dots , k  \}$.  Then
\[
0\leq 2c_{i_1} \leq 2k \qquad \qquad \text{ and }  \qquad \qquad 0 \leq c_{i_2} + c_{i_3} \leq 2k \qquad \qquad \text{ for any }\,  c_{i_1}, c_{i_2}, c_{i_3} \in \{ 0, 1, 2, \dots , k  \}.
\]
Let us then consider the triples $(c_{i_1}, c_{i_2}, c_{i_3})$ for which the coloring condition holds at a crossing of the diagram i.e. for which
\[
2c_{i_1} =  c_{i_2} + c_{i_3} \qquad \qquad \text{ mod }  2k+1.
\]
Since each side of these equalities takes on a value between $0$ and $2k$ then each of these equalities also holds over the integers.

This further implies that these coloring conditions hold for any prime which means that the link $L$ has non-trivial colorings modulo infinitely many primes. But this conflicts with the standing hypothesis that $\det L \neq 0$, thus completing the proof.
\end{proof}

\begin{cor}\label{cor:s-consec}
Let $k$  be a positive integer. Let $f$ be a $(2k+1)$-coloring automorphism. Let $L$ be a link with non-null determinant, admitting non-trivial $(2k+1)$-colorings.

A set $S\subseteq \{ f(0), f(1), \dots \, , f(k) \}$ cannot be a $(2k+1)$-Sufficient Set of Colors for $L$.
\end{cor}
\begin{proof}
Assume to the contrary and suppose there is a non-trivial $(2k+1)$-coloring consisting of colors from the set:
\[
\{ f(0), f(1), f(2), \cdots , f(k)  \}.
\]
Let $g$ be the inverse permutation to $f$. Then colors from the set
\[
\{ g(f(0)), g(f(1)), g(f(2)), \cdots , g(f(k)) \} = \{ 0, 1, 2, \cdots , k  \}
\]
also give rise to a non-trivial $(2k+1)$-coloring which contradicts Theorem \ref{thm:lessk}. This completes the proof.
\end{proof}

\subsection{Further Obstructions.}

 The relevance of Theorem \ref{thm:lessk} and Corollary \ref{cor:s-consec} is the following.
 A possible way of reducing the number of colors is to start from a diagram of the link under study endowed with a non-trivial $p$-coloring and to use Reidemeister moves followed by local reassignment of colors, to obtain a new diagram of the same link endowed with a non-trivial $p$-coloring using a proper subset of the set of colors used by the former coloring. Theorem \ref{thm:lessk} and Corollary \ref{cor:s-consec} may help recognizing which colors of the former coloring may be eliminated. These results indicate which colors cannot be removed; the remaining ones are the candidates to colors to be eliminated.

Theorem \ref{thm:kauffmansaito} provides us with a subsequent test on these candidates.

\begin{thm}\label{thm:kauffmansaito}$\cite{KauffmanChicago2004, Saito}$
Let $k$  be a positive integer.  Let $L$ be a link of non-null determinant and admitting non-trivial $(2k+1)$-colorings.

Suppose $S= \{   c_1, \dots , c_n   \}$ is a $(2k+1)$-Sufficient Set of Colors for $L$. $($Without loss of generality we take the $c_i$'s from $\{ 0, 1, 2, \dots \, , 2k \}$.$)$ For each $c_i$ $(i=1, 2, \dots , n)$, let $S_i$ be the set of unordered pairs $\{ c_i^1, c_i^2 \}$ $($with $c_i^1\neq c_i^2 \text{ from } S$$)$ such that $2c_i=c_i^1+c_i^2$ mod $2k+1$.

\begin{enumerate}
\item There is at least one $i$ such that the expression $2c_i=c_i^1+c_i^2$ does not make sense over the integers, it only makes sense modulo $2k+1$.

\item If there is $i\in \{ 1, 2, \dots , n \}$ such that, for all $j\in \{ 1, 2, \dots , n  \}$, $c_i$ does not belong to any of the unordered pairs in $S_j$, then $n> mincol_p\, L$ and  $S\setminus \{ c_i \}$ is also a $(2k+1)$-Sufficient Set of Colors.
\end{enumerate}
\end{thm}
\begin{proof}
\begin{enumerate}
\item If for all $i$'s each of the  expressions $2c_i=c_i^1+c_i^2$ makes sense without considering it mod $2k+1$, then we have colorings whose coloring conditions hold modulo any prime. But this conflicts with the $\det L \neq 0$ hypothesis which implies that $\det L$ admits only a finite number of prime divisors.
\item Since $c_i$ cannot be the color of an under-arc at a polichromatic crossing, then this implies
that $L$ is a split link such that one or more link components are colored with $c_i$
alone, which conflicts with the hypothesis in the statement, that links have non-null determinant.
\end{enumerate}
This completes the proof.
\end{proof}

\noindent

\section{Removing Colors From Non-Trivial $11$-Colorings (Knots $10_{125}$, $10_{128}$, and $10_{152}$) and From Non-Trivial $13$-Colorings (Knots $6_3$, $7_3$, $8_1$, $9_{43}$, and $10_{154}$).}\label{sect:11-13}
\noindent
In this Section we apply  the obstruction results of Section \ref{sect:obst}: Theorem \ref{thm:lessk} and Corollary \ref{cor:s-consec}, and Theorem \ref{thm:kauffmansaito}. We remark that these are examples where this procedure yields results but that it may not work in general. The original diagrams are adapted from \cite{Rolfsen}, except for the original diagram for $8_1$ which is a braid closure whose braid word was obtained in \cite{VJones}. We obtain the following results.
\[
5 = mincol_{11}10_{125} = mincol_{11}10_{128}  = mincol_{11}10_{152}  \qquad (11 = \det 10_{125} = \det 10_{128} = \det 10_{152}).
\]
\begin{align*}
&5 = mincol_{13}6_{3} = mincol_{13}7_{3} = mincol_{13} 8_1 = mincol_{13} 9_{43} = mincol_{13}10_{154} \\
&(13 = \det 6_{3} = \det 7_{3} = \det 8_{1} = \det 9_{43} = \det 10_{154}).
\end{align*}

\begin{figure}[!ht]
	\psfrag{0}{\huge$0$}
	\psfrag{1}{\huge$1$}
	\psfrag{2}{\huge$2$}
	\psfrag{3}{\huge$3$}
	\psfrag{4}{\huge$4$}
	\psfrag{5}{\huge$5$}
	\psfrag{6}{\huge$6$}
	\psfrag{8}{\huge$8$}
	\psfrag{10}{\huge$10$}
	\psfrag{11}{\huge$\mathbf{11}$}
	\centerline{\scalebox{.33}{\includegraphics{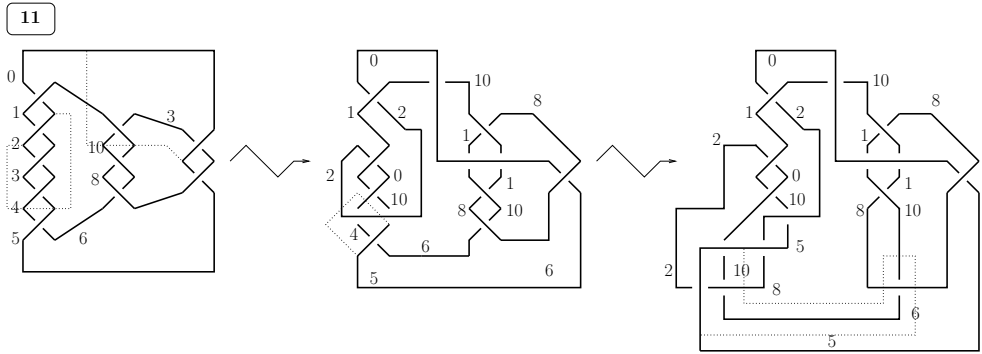}}}
	\caption{The knot $10_{125}$ whose determinant is $11$.   The dotted lines indicate the move that will take to the next diagram. The sequence of Figures \ref{fig:10-125}, \ref{fig:10-125cont}, and \ref{fig:10-125contcont} shows that $mincol_{11}\, 10_{125} = 5$ along with the set of $5$ colors mod $11$, $\{ 0, 1, 5, 8, 10 \}$, that color the final diagram.}\label{fig:10-125}
\end{figure}
\begin{figure}[!ht]
	\psfrag{0}{\huge$0$}
	\psfrag{1}{\huge$1$}
	\psfrag{2}{\huge$2$}
	\psfrag{3}{\huge$3$}
	\psfrag{4}{\huge$4$}
	\psfrag{5}{\huge$5$}
	\psfrag{6}{\huge$6$}
	\psfrag{8}{\huge$8$}
	\psfrag{10}{\huge$10$}
	\psfrag{11}{\huge$\mathbf{11}$}
	\centerline{\scalebox{.33}{\includegraphics{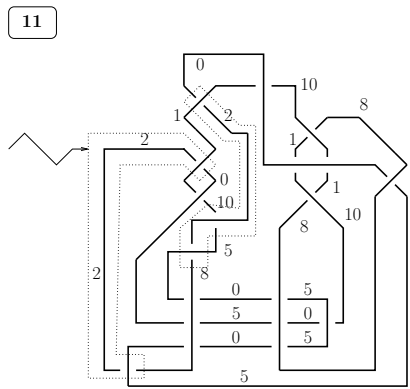}}}
	\caption{The knot $10_{125}$ whose determinant is $11$.   Continued from Figure \ref{fig:10-125}. The regions boxed with dotted lines are treated in Figure \ref{fig:10-125contcont}.}\label{fig:10-125cont}
\end{figure}
\begin{figure}[!ht]
	\psfrag{0}{\huge$0$}
	\psfrag{1}{\huge$1$}
	\psfrag{2}{\huge$2$}
	\psfrag{3}{\huge$3$}
	\psfrag{4}{\huge$4$}
	\psfrag{5}{\huge$5$}
	\psfrag{6}{\huge$6$}
	\psfrag{8}{\huge$8$}
	\psfrag{10}{\huge$10$}
	\psfrag{11}{\huge$\mathbf{11}$}
	\centerline{\scalebox{.33}{\includegraphics{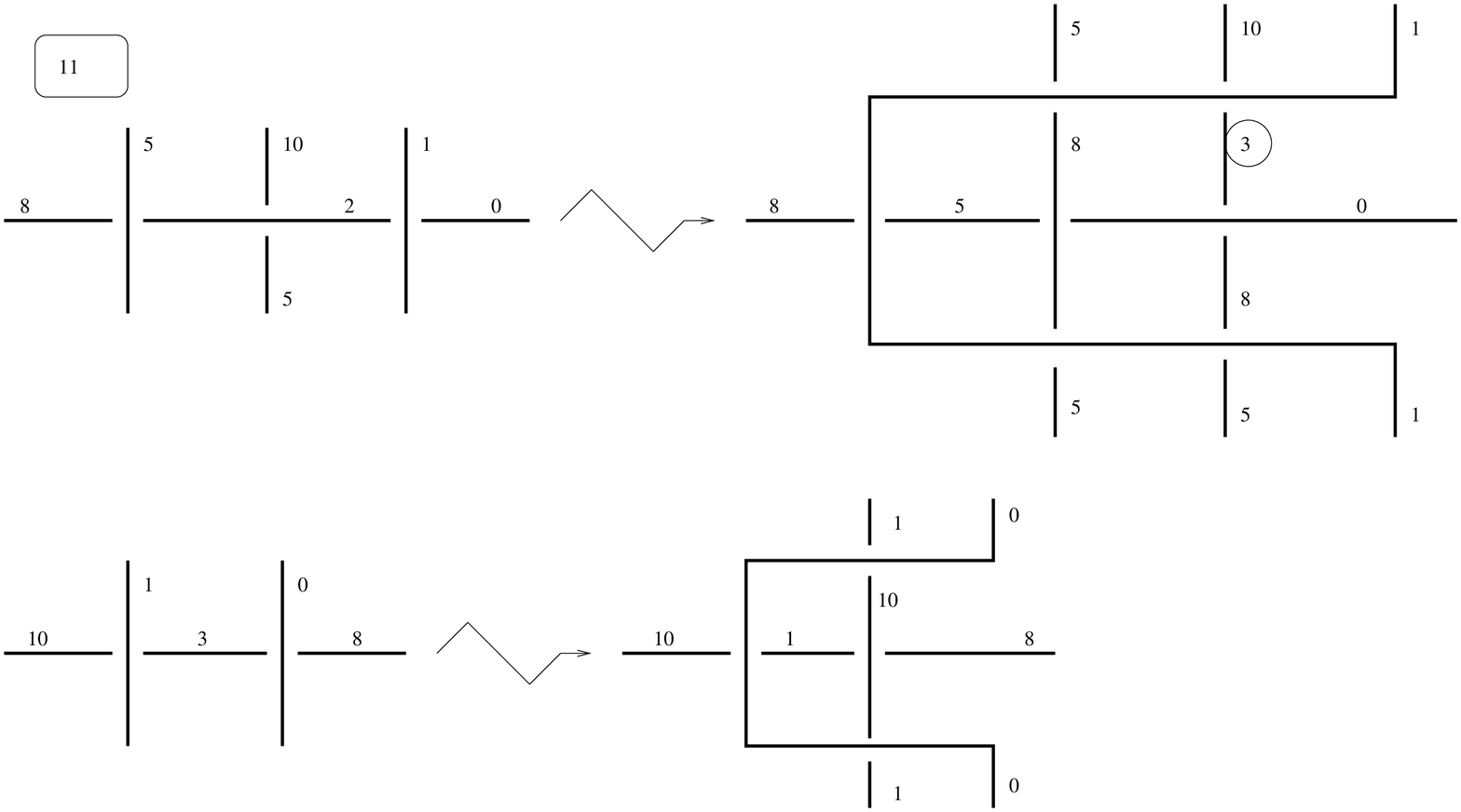}}}
	\caption{The knot $10_{125}$ whose determinant is $11$, conclusion.   The circle around the $3$ tells us it still has to be removed; this is done below in this Figure. With this we accomplish coloring $10_{125}$ mod $11$ with $5$ colors, $\{ 0, 1, 5, 8, 10 \}$. We refrain from encorporating these adjustments  into Figure \ref{fig:10-125cont} in order not to over-burden things. We note that the $11$-coloring automorphism $f(x)=8x$ maps $\{ 0, 1, 2, 4, 7 \}$ onto $\{ 0, 1, 5, 8, 10 \}$; the $11$-coloring automorphism $g(x)=2x$ maps $\{ 0, 1, 2, 4, 7 \}$ onto $\{ 0, 2, 3, 4, 8 \}$.}\label{fig:10-125contcont}
\end{figure}
\begin{figure}[!ht]
	\psfrag{0}{\huge$0$}
	\psfrag{1}{\huge$1$}
	\psfrag{2}{\huge$2$}
	\psfrag{3}{\huge$3$}
	\psfrag{4}{\huge$4$}
	\psfrag{5}{\huge$5$}
	\psfrag{6}{\huge$6$}
	\psfrag{8}{\huge$8$}
	\psfrag{9}{\huge$9$}
	\psfrag{10}{\huge$10$}
	\psfrag{11}{\huge$\mathbf{11}$}
	\centerline{\scalebox{.33}{\includegraphics{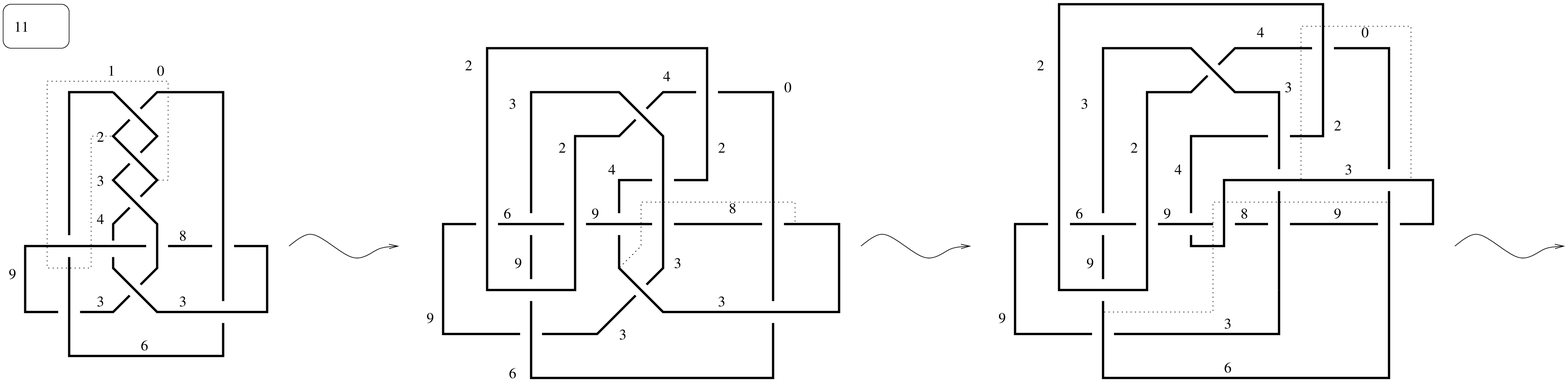}}}
	\caption{The knot $10_{128}$ whose determinant is $11$.   The dotted lines indicate the move that will take to the next diagram. }\label{fig:10-128}
\end{figure}
\begin{figure}[!ht]
	\psfrag{0}{\huge$0$}
	\psfrag{1}{\huge$1$}
	\psfrag{2}{\huge$2$}
	\psfrag{3}{\huge$3$}
	\psfrag{4}{\huge$4$}
	\psfrag{5}{\huge$5$}
	\psfrag{6}{\huge$6$}
	\psfrag{8}{\huge$8$}
	\psfrag{9}{\huge$9$}
	\psfrag{10}{\huge$10$}
	\psfrag{11}{\huge$\mathbf{11}$}
	\centerline{\scalebox{.33}{\includegraphics{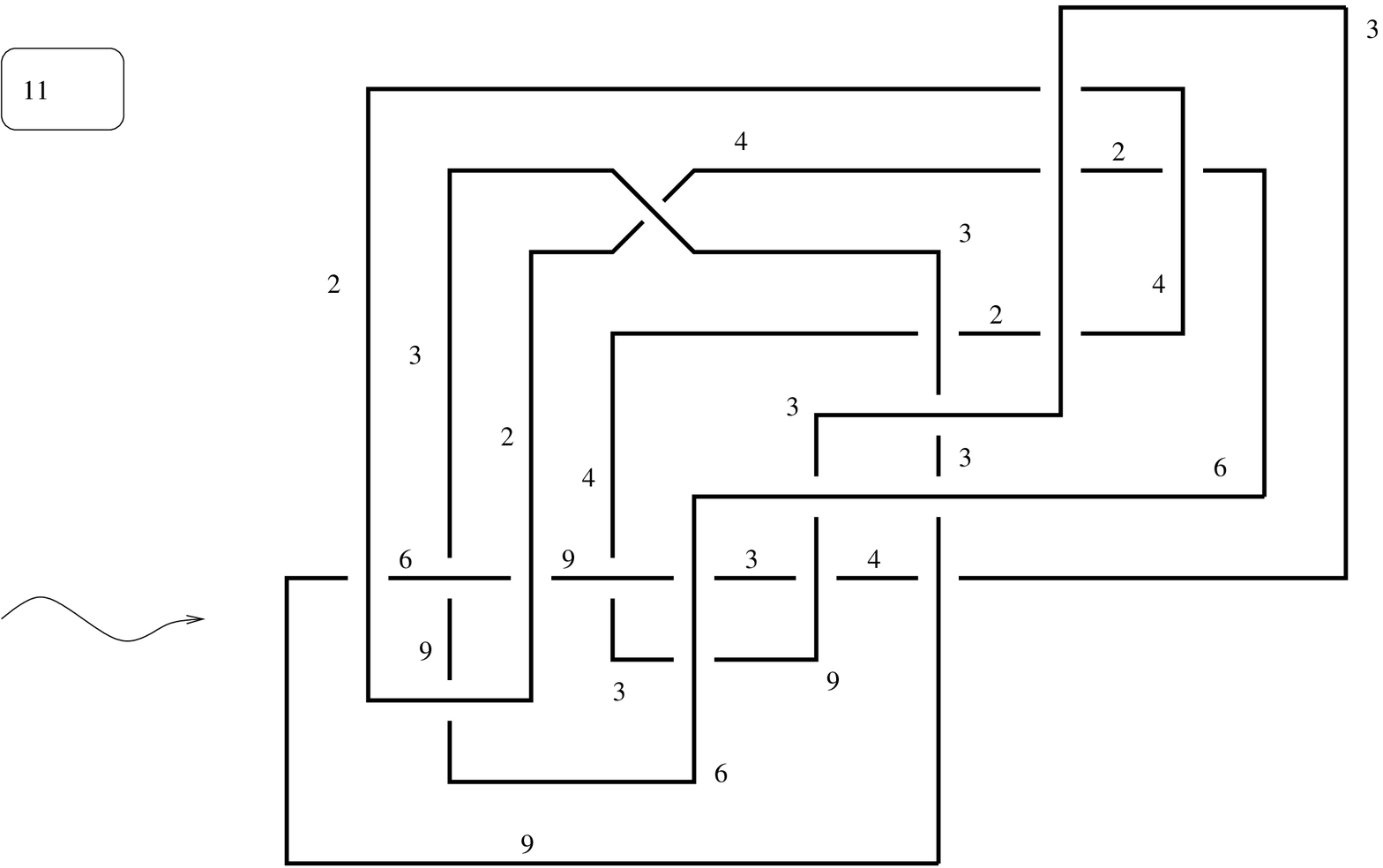}}}
	\caption{The knot $10_{128}$ whose determinant is $11$, conclusion. An $11$-coloring using colors $\{ 2, 3, 4, 6, 9  \}$ is obtained. Thus $mincol_{11}\, 10_{128}=5$. The $11$-coloring automorphism $f(x)=2x-4$ takes this set to $\{ 0, 2, 3, 4, 8 \}$.}\label{fig:10-128concl}
\end{figure}
\begin{figure}[!ht]
	\psfrag{0}{\huge$0$}
	\psfrag{1}{\huge$1$}
	\psfrag{2}{\huge$2$}
	\psfrag{3}{\huge$3$}
	\psfrag{4}{\huge$4$}
	\psfrag{5}{\huge$5$}
	\psfrag{6}{\huge$6$}
	\psfrag{8}{\huge$8$}
	\psfrag{10}{\huge$10$}
	\psfrag{11}{\huge$\mathbf{11}$}
	\centerline{\scalebox{.33}{\includegraphics{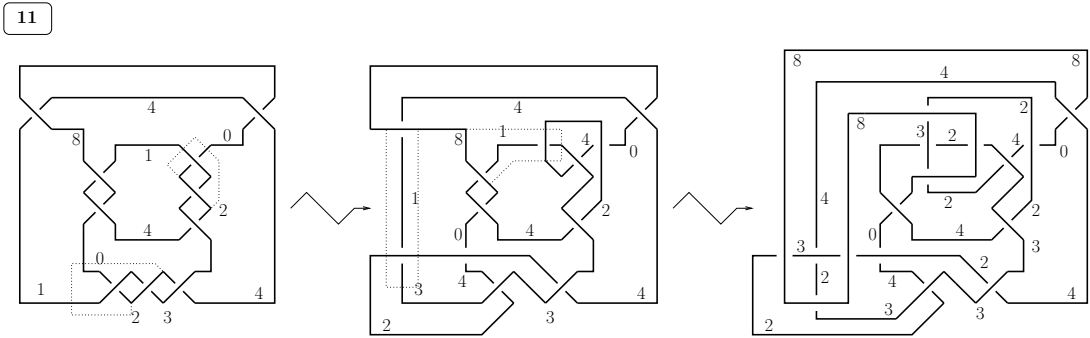}}}
	\caption{The knot $10_{152}$ whose determinant is $11$.  The dotted lines indicate the move that will take to the next diagram. A non-trivial $11$-coloring of $10_{152}$ is obtained with colors $\{ 0, 2, 3, 4, 8\}$. Thus $mincol_{11}\, 10_{152}=5$. The $11$-coloring automorphism $f(x)=2x$ maps the set $\{ 0, 1, 2, 4, 7\}$ to the set $\{ 0, 2, 3, 4, 8 \}$.}\label{fig:10-152}
\end{figure}

\begin{figure}[!ht]
	\psfrag{0}{\huge$0$}
	\psfrag{1}{\huge$1$}
	\psfrag{2}{\huge$2$}
	\psfrag{3}{\huge$3$}
	\psfrag{4}{\huge$4$}
	\psfrag{8}{\huge$8$}
	\psfrag{6}{\huge$6$}
	\psfrag{11}{\huge$11$}
	\psfrag{13}{\huge$\mathbf{13}$}
	\centerline{\scalebox{.33}{\includegraphics{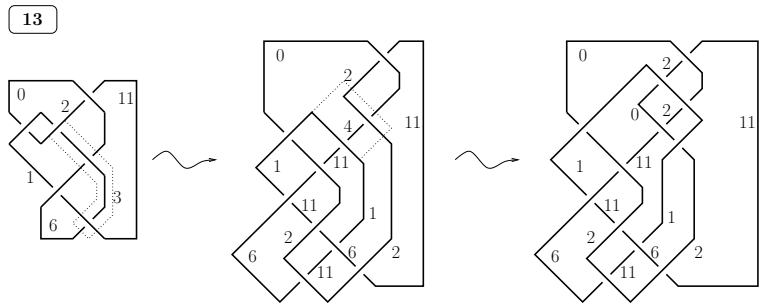}}}
	\caption{The knot $6_3$ whose determinant is $13$.   On the left-hand side, a diagram with minimum number of crossings equipped with a non-trivial $13$-coloring; the dotted lines indicate the move that will take to the diagram on the right-hand side. On the right-hand side a diagram equipped with a non-trivial $13$-coloring using a minimum number of colors modulo $13$: $\{ 0, 1, 2, 6, 11 \}$. Thus $mincol_{13}\, 6_{3}=5$.}\label{fig:6-3}
\end{figure}
\begin{figure}[!ht]
	\psfrag{0}{\huge$0$}
	\psfrag{1}{\huge$1$}
	\psfrag{2}{\huge$2$}
	\psfrag{3}{\huge$3$}
	\psfrag{4}{\huge$4$}
	\psfrag{5}{\huge$5$}
	\psfrag{9}{\huge$9$}
	\psfrag{11}{\huge$11$}
	\psfrag{13}{\huge$\mathbf{13}$}
	\centerline{\scalebox{.33}{\includegraphics{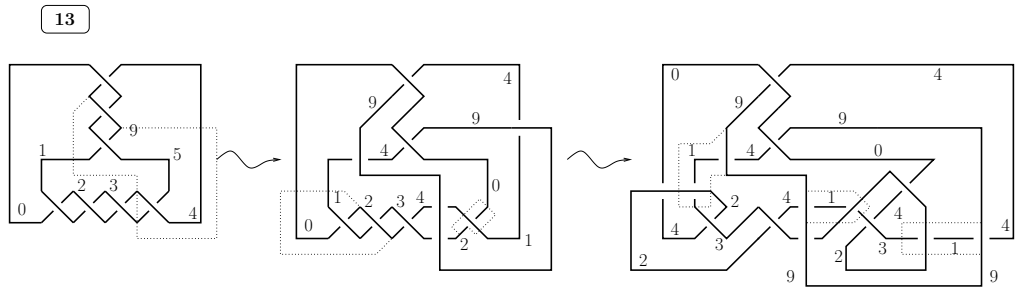}}}
	\caption{The knot $7_3$ whose determinant is $13$.   On the left-hand side, a diagram with minimum number of crossings equipped with a non-trivial $13$-coloring; the dotted lines indicate the move that will take to the next diagram.}\label{fig:7-3}
\end{figure}
\begin{figure}[!ht]
	\psfrag{0}{\huge$0$}
	\psfrag{1}{\huge$1$}
	\psfrag{2}{\huge$2$}
	\psfrag{3}{\huge$3$}
	\psfrag{4}{\huge$4$}
	\psfrag{5}{\huge$5$}
	\psfrag{9}{\huge$9$}
	\psfrag{11}{\huge$11$}
	\psfrag{13}{\huge$\mathbf{13}$}
	\centerline{\scalebox{.33}{\includegraphics{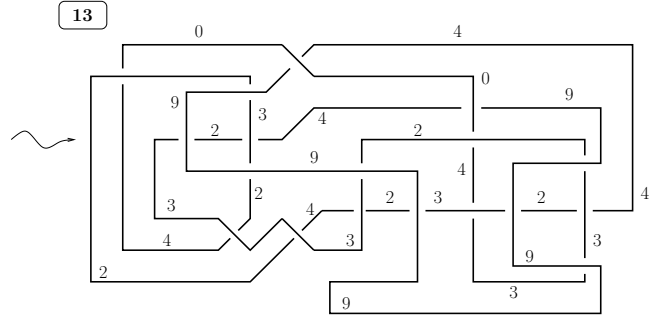}}}
	\caption{The knot $7_3$ whose determinant is $13$, conclusion.   A non-trivial $13$-coloring of $7_3$ here with colors $\{ 0, 2, 3, 4, 9 \}$ is obtained. Thus $mincol_{13}\, 7_{3}=5$. Also, the $13$-coloring automorphism $f(x)=6x+3$ maps the set of colors  $\{ 0, 1, 2, 6, 11 \}$ to the set $\{ 0, 2, 3, 4, 9 \}$. }\label{fig:7-3concl}
\end{figure}
\begin{figure}[!ht]
	\psfrag{0}{\huge$0$}
	\psfrag{1}{\huge$1$}
	\psfrag{2}{\huge$2$}
	\psfrag{3}{\huge$3$}
	\psfrag{4}{\huge$4$}
	\psfrag{5}{\huge$5$}
	\psfrag{10}{\huge$10$}
	\psfrag{13}{\huge$\mathbf{13}$}
	\centerline{\scalebox{.33}{\includegraphics{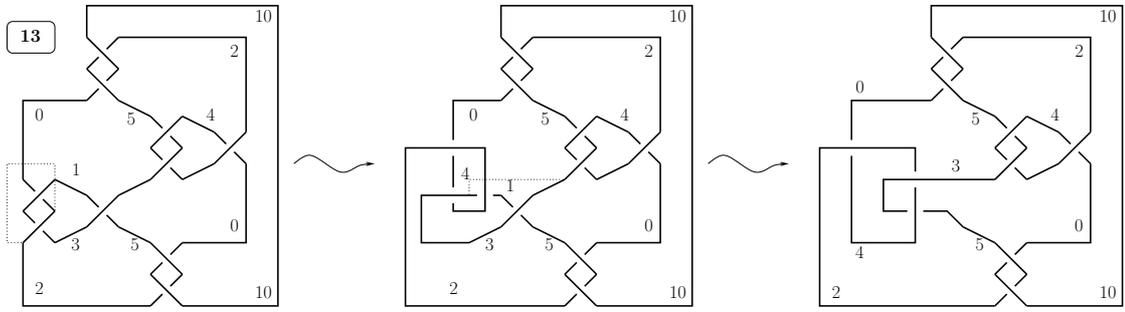}}}
	\caption{The knot $10_{154}$ whose determinant is $13$.   A minimum number of $5$ colors is needed to obtain a non-trivial $13$-coloring of $10_{154}$. In Figure \ref{fig:10-154cont} we show how to remove color $2$ thus showing $10_{154}$ can be colored with a minimum of $5$ colors mod $13$, $\{ 0, 3, 4, 5, 10 \}$. }\label{fig:10-154}
\end{figure}
\begin{figure}[!ht]
	\psfrag{0}{\huge$0$}
	\psfrag{1}{\huge$1$}
	\psfrag{2}{\huge$2$}
	\psfrag{3}{\huge$3$}
	\psfrag{4}{\huge$4$}
	\psfrag{5}{\huge$5$}
	\psfrag{7}{\huge$7$}
	\psfrag{10}{\huge$10$}
	\psfrag{13}{\huge$\mathbf{13}$}
	\centerline{\scalebox{.33}{\includegraphics{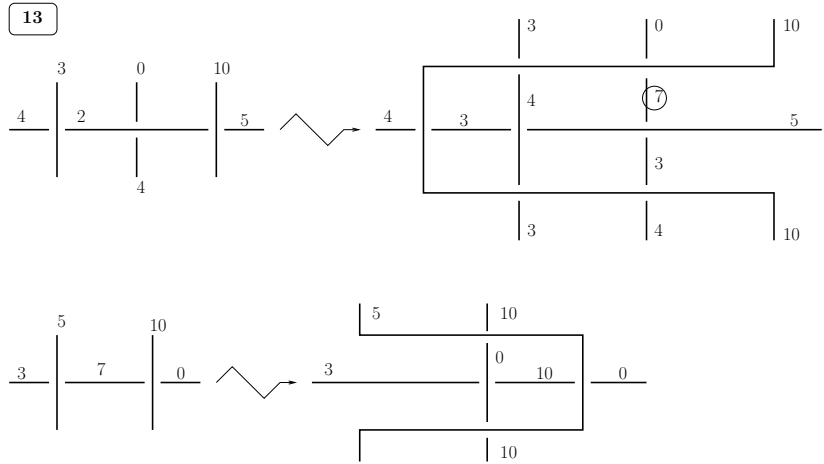}}}
	\caption{The knot $10_{154}$ whose determinant is $13$, conclusion.   The circle around the $7$ tells us it still has to be removed; this is done below in this figure. With this we accomplish coloring $10_{154}$ mod $13$ with $5$ colors, $\{ 0, 3, 4, 5, 10 \}$. Thus $mincol_{13}\, 10_{154}=5$. We refrain form encorporating the adjustments in this figure in Figure \ref{fig:10-154} in order not to over-burden things. Also, the $13$-coloring automorphism $f(x)=5x$ maps the set of colors $\{ 0, 1, 2, 6, 11 \}$ to the set  $\{ 0, 3, 4, 5, 10 \}$. }\label{fig:10-154cont}
\end{figure}

\begin{figure}[!ht]
	\psfrag{0}{\huge$0$}
	\psfrag{1}{\huge$1$}
	\psfrag{2}{\huge$2$}
	\psfrag{3}{\huge$3$}
	\psfrag{4}{\huge$4$}
	\psfrag{5}{\huge$5$}
	\psfrag{7}{\huge$7$}
	\psfrag{6}{\huge$6$}
	\psfrag{8}{\huge$8$}
	\psfrag{10}{\huge$10$}
	\psfrag{11}{\huge$11$}
	\psfrag{12}{\huge$12$}
	\psfrag{13}{\huge$\mathbf{13}$}
	\centerline{\scalebox{.35}{\includegraphics{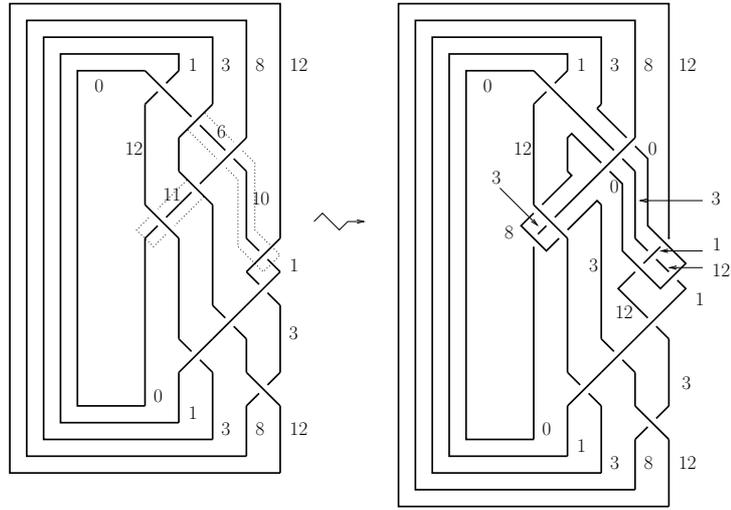}}}
	\caption{The knot $8_1$ whose determinant is $13$.    The dotted lines indicate the move that will take to the next diagram. Modulo $13$ we are able to color $8_1$ with the following $5$ colors: $\{ 0, 1, 3, 8, 12 \}$. Thus $mincol_{13}\, 8_{1}=5$. $f(x)=-x+1$, mod $13$, maps $\{ 0, 1, 3, 8, 12 \}$ into $\{ 0, 1, 2, 6, 11  \}$.}\label{fig:8-1new}
\end{figure}

\begin{figure}[!ht]
	\psfrag{0}{\huge$0$}
	\psfrag{1}{\huge$1$}
	\psfrag{2}{\huge$2$}
	\psfrag{3}{\huge$3$}
	\psfrag{4}{\huge$4$}
	\psfrag{5}{\huge$5$}
	\psfrag{7}{\huge$7$}
	\psfrag{6}{\huge$6$}
	\psfrag{8}{\huge$8$}
	\psfrag{11}{\huge$11$}
	\psfrag{12}{\huge$12$}
	\psfrag{13}{\huge$\mathbf{13}$}
	\centerline{\scalebox{.3}{\includegraphics{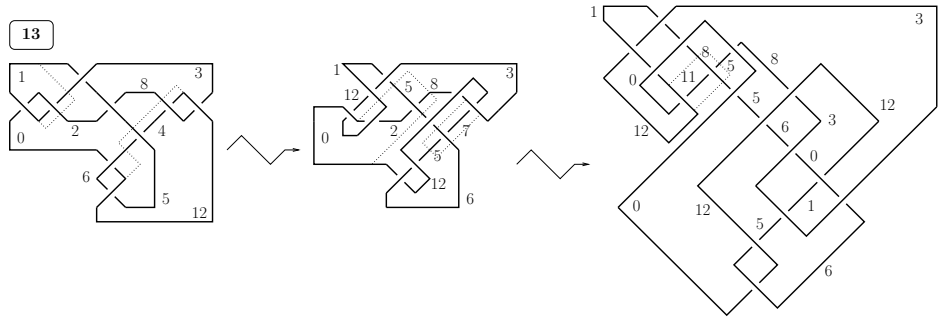}}}
	\caption{The knot $9_{43}$ whose determinant is $13$.    The dotted lines indicate the move that will take to the next diagram. }\label{fig:9-43}
\end{figure}
\begin{figure}[!ht]
	\psfrag{0}{\huge$0$}
	\psfrag{1}{\huge$1$}
	\psfrag{2}{\huge$2$}
	\psfrag{3}{\huge$3$}
	\psfrag{4}{\huge$4$}
	\psfrag{5}{\huge$5$}
	\psfrag{7}{\huge$7$}
	\psfrag{6}{\huge$6$}
	\psfrag{8}{\huge$8$}
	\psfrag{11}{\huge$11$}
	\psfrag{12}{\huge$12$}
	\psfrag{13}{\huge$\mathbf{13}$}
	\centerline{\scalebox{.27}{\includegraphics{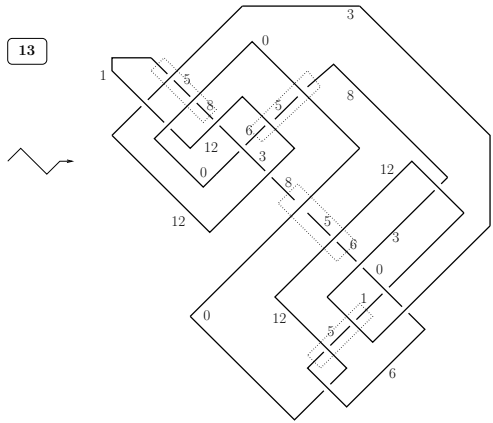}}}
	\caption{The knot $9_{43}$ whose determinant is $13$. The boxed regions involve color $5$ or colors $5$ and $6$; there are three distinct cases. In  Figures \ref{fig:9-43I}, \ref{fig:9-43IIv2}, \ref{fig:9-43IIv2b}, and \ref{fig:9-43IIIv2} we show how to eliminate them. After eliminating these colors we end up with a non-trivial $13$-coloring of $9_{43}$ using colors $0, 1, 3, 8, 12$ mod $13$. Hence $mincol_{13}\,9_{43}=5$.}\label{fig:9-43cont}
\end{figure}
\begin{figure}[!ht]
	\psfrag{0}{\huge$0$}
	\psfrag{1}{\huge$1$}
	\psfrag{2}{\huge$2$}
	\psfrag{3}{\huge$3$}
	\psfrag{4}{\huge$4$}
	\psfrag{5}{\huge$5$}
	\psfrag{7}{\huge$7$}
	\psfrag{6}{\huge$6$}
	\psfrag{8}{\huge$8$}
	\psfrag{10}{\huge$10$}
	\psfrag{12}{\huge$12$}
	\psfrag{13}{\huge$\mathbf{13}$}
	\centerline{\scalebox{.3}{\includegraphics{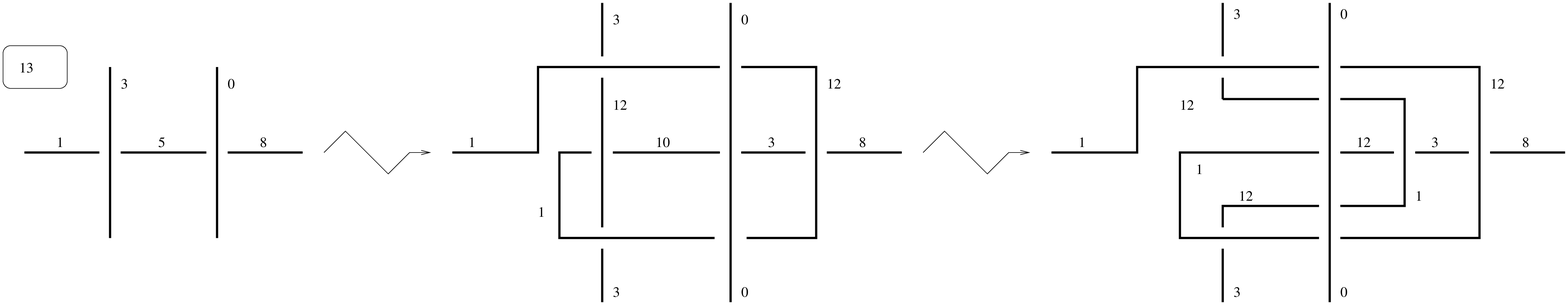}}}
	\caption{The knot $9_{43}$ whose determinant is $13$. Taking care of the $5$ in one of the boxed regions in Figure \ref{fig:9-43cont}; first case.}\label{fig:9-43I}
\end{figure}
\begin{figure}[!ht]
	\psfrag{0}{\huge$0$}
	\psfrag{1}{\huge$1$}
	\psfrag{2}{\huge$2$}
	\psfrag{3}{\huge$3$}
	\psfrag{4}{\huge$4$}
	\psfrag{5}{\huge$5$}
	\psfrag{7}{\huge$7$}
	\psfrag{6}{\huge$6$}
	\psfrag{8}{\huge$8$}
	\psfrag{10}{\huge$10$}
	\psfrag{12}{\huge$12$}
	\psfrag{13}{\huge$\mathbf{13}$}
	\centerline{\scalebox{.25}{\includegraphics{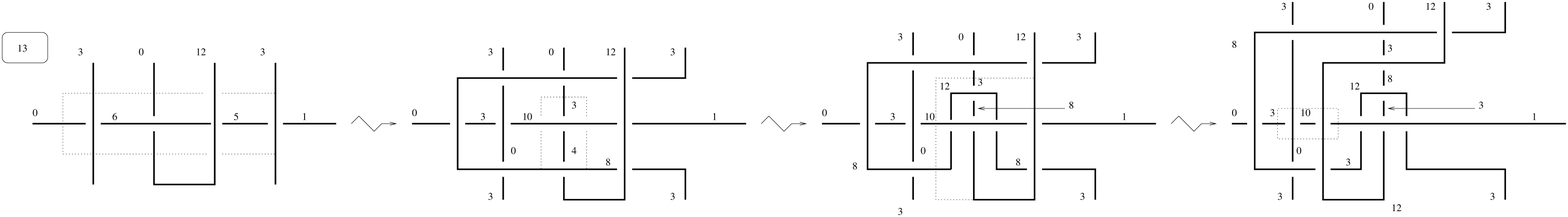}}}
	\caption{The knot $9_{43}$ whose determinant is $13$. Taking care of the $5$ in one of the boxed regions in Figure \ref{fig:9-43cont}; second case.}\label{fig:9-43IIv2}
\end{figure}
\begin{figure}[!ht]
	\psfrag{0}{\huge$0$}
	\psfrag{1}{\huge$1$}
	\psfrag{2}{\huge$2$}
	\psfrag{3}{\huge$3$}
	\psfrag{4}{\huge$4$}
	\psfrag{5}{\huge$5$}
	\psfrag{7}{\huge$7$}
	\psfrag{6}{\huge$6$}
	\psfrag{8}{\huge$8$}
	\psfrag{10}{\huge$10$}
	\psfrag{12}{\huge$12$}
	\psfrag{13}{\huge$\mathbf{13}$}
	\centerline{\scalebox{.3}{\includegraphics{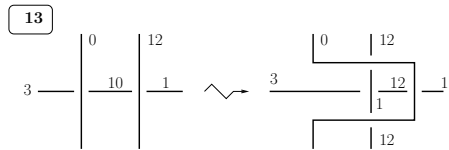}}}
	\caption{The knot $9_{43}$ whose determinant is $13$. Taking care of the $10$ in the boxed region in the right-most diagram in Figure \ref{fig:9-43IIv2}.}\label{fig:9-43IIv2b}
\end{figure}
\begin{figure}[!ht]
	\psfrag{0}{\huge$0$}
	\psfrag{1}{\huge$1$}
	\psfrag{2}{\huge$2$}
	\psfrag{3}{\huge$3$}
	\psfrag{4}{\huge$4$}
	\psfrag{5}{\huge$5$}
	\psfrag{7}{\huge$7$}
	\psfrag{6}{\huge$6$}
	\psfrag{8}{\huge$8$}
	\psfrag{11}{\huge$11$}
	\psfrag{12}{\huge$12$}
	\psfrag{13}{\huge$\mathbf{13}$}
	\centerline{\scalebox{.3}{\includegraphics{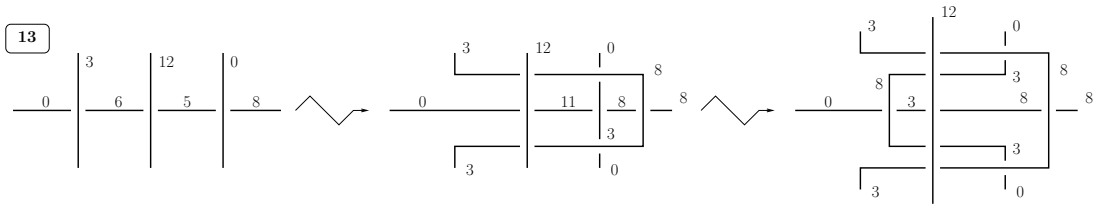}}}
	\caption{The knot $9_{43}$ whose determinant is $13$. Taking care of the $5$ and the $6$ in one of the boxed regions in Figure \ref{fig:9-43cont}; third case. We refrain from including these treatments in the diagram of Figure \ref{fig:9-43cont} in order not to overburden it. This concludes our attempt at minimizing the number of colors for a non-trivial $13$-coloring of $9_{43}$: we did it with $5$ colors, $\{ 0, 1, 3, 8, 12  \}$.}\label{fig:9-43IIIv2}
\end{figure}

\section{Directions for Future Work}\label{sect:future}

\noindent

We distinguish three main topics to be looked into: $1.$ $11$-minimal sufficient set of colors; $2.$ minimum number of colors; $3.$ Procedure for reducing the number of colors.

\begin{enumerate}
\item We showed that there is no universal $11$-minimal sufficient set of colors leaning on the fact that there are diagrams of $6_2$ and $7_2$ which support minimal $11$-colorings but whose $11$-minimal sufficient sets of colors cannot be mapped into each other via an $11$-coloring automorphism. The following questions seem to be in order at this point. Do there exist two diagrams, one for each of these knots, which could be colored by the same $11$-minimal sufficient set of colors? Could there be Reidemeister moves followed by local rearrangement of colors which would take the diagram of $6_2$ in the right-hand side of Figure \ref{fig:6-2} (colored with $\{ 0, 2, 3, 4, 8 \}$) into a diagram colored with $\{ 0, 3, 4, 5, 6 \}$? What about for other prime moduli $p > 11$?

\item Given a prime $p(=2k+1)$, if there are links $L$ and $L'$ such that $mincol_p\, L \neq mincol_p\, L'$, what is the range of the function $mincol_p (\dots)$? Is there a link $L$ such that $mincol_p\, L > k$?

\item This procedure is as follows. Given a prime $p$, and given a diagram $D$, endowed with a non-trivial $p$-coloring let $S = \{  c_1, \dots , c_n \}$ (mod $p$) be the set of colors used in the coloring. Use Theorem \ref{thm:lessk} and Corollary \ref{cor:s-consec} to list the candidates to $p$-Sufficient Sets of Colors of the form $S\setminus \{ c_i \}$. Use Theorem \ref{thm:kauffmansaito} to further screen this list. Use the list just obtained to guide in the choice of the next color to be eliminated from the coloring.

    We remark that this is a procedure and not an algorithm: conceivably certain diagrams equipped with non-trivial colorings may not allow a reduction of the number of colors by this procedure.

\end{enumerate}

We plan to look into these issues in the near future.

\end{document}